\documentclass{article}
\usepackage[]{graphicx}
\usepackage[]{amsmath,amssymb}
\usepackage[]{amsthm,enumerate}
\usepackage{verbatim,bm,comment}
\usepackage{color}

\newtheorem{theorem}{Theorem}[section]
\newtheorem{lemma}[theorem]{Lemma}
\newtheorem{proposition}[theorem]{Proposition}

\theoremstyle{definition}

\newtheorem{example}[theorem]{Example}

\theoremstyle{remark}
\newtheorem{remark}[theorem]{Remark}

\title{
Hoffman's coclique bound for normal regular digraphs, and nonsymmetric association schemes. 
} 
\date{
\today
}

\author{
 Hadi Kharaghani\thanks{Department of Mathematics and Computer Science, University of Lethbridge,
Lethbridge, Alberta, T1K 3M4, Canada. \texttt{kharaghani@uleth.ca}} 
\and  
 Sho Suda\thanks{Department of Mathematics Education,  Aichi University of Education, 1 Hirosawa, Igaya-cho, Kariya, Aichi 448-8542, Japan. \texttt{suda@auecc.aichi-edu.ac.jp}}
}

\begin{document}

\maketitle
\begin{abstract}

We extend Hoffman's coclique bound for regular digraphs with the property that its adjacency matrix is normal, and discuss cocliques attaining the inequality.
As a consequence, we characterize skew-Bush-type Hadamard matrices in terms of digraphs.  
We present some normal digraphs whose vertex set is decomposed into disjoint cocliques attaining the bound.  
The digraphs provided here are relation graphs of some nonsymmetric association schemes. 
\end{abstract}

\section{Introduction}
Spectral graph theory for undirected graph has been studied very well \cite{BH}. 
Using the eigenvalues of the adjacency matrix of a graph, we obtain several inequalities for parameters of the graph, such as the clique number,  the independence number, the chromatic number, etc. 
Hoffman gave an upper bound for independence number of regular graphs to use the eigenvalues of the adjacency matrix. 

In this paper we extend the Hoffman bound for normal regular digraphs. Here, a normal digraph means a digraph with the adjacency matrix being normal, see Section~\ref{sec:digraph}. 
A coclique attaining the upper bound is also studied. 
In Section~\ref{sec:sbh}, we use the bound to characterize skew-Bush-type (or skew-checkered) Hadamard matrices in terms of doubly regular asymmetric digraphs with some properties. 
This result is an analogy of the result by Wallis \cite{W} that there exists a symmetric Bush-type Hadamard matrix of order $4n^2$ if and only if a strongly regular graph with parameters $(4n^2,2n^2-n,n^2-n,n^2-n)$ such that the vertex set is decomposed into $2n$ disjoint cocliques of size $2n$.  
Note that the cocliques of size $2n$ in the strongly regular graph attain Hoffman's bound. 

A coclique attaining Hoffman's coclique bound in a strongly regular graph $\Gamma$ is a clique attaining the clique bound in the complement of $\Gamma$. 
A spread of a strongly regular graph is a set of disjoint cliques attaining the clique bound.
In \cite{B,C,HT}, a spread of strongly regular graphs is extensively studied.
In Section~\ref{sec:HL}, \ref{sec:td} we provide some normal digraphs with the vertex set decomposed into disjoint cocliques attaining the upper bound. 
All of them are relation graphs of association schemes. 

\section{Preliminaries}
\subsection{Digraphs}\label{sec:digraph}
A {\em digraph} $\Gamma$ is a pair $(X,E)$ such that the {\em vertex set} $X$ is a finite set and the {\em edge set} or {\em arc set} $E$ is a subset of $X\times X$ with $E\cap \{(x,x)\mid x\in X\}=\emptyset$.
The {\em adjacency matrix} of $\Gamma$ is a $(0,1)$-matrix with rows and columns indexed by the elements of $X$ such that $A_{xy}=1$ if $(x,y)\in E$ and $A_{xy}=0$ otherwise. 
A digraph $\Gamma$ is {\em asymmetric} if $(x,y)\in E$ implies $(y,x)\not\in E$, namely $A+A^T$ is a $(0,1)$-matrix, where $A^T$ denotes the transpose of $A$. 
A digraph $\Gamma$ is {\em normal} if the adjacency matrix $A$ is normal, namely $AA^T=A^TA$ holds. 
A digraph $\Gamma$ is {\em $k$-regular} if $|\{y\in X\mid (x,y)\in E\}|=|\{y\in X\mid (y,x)\in E\}|=k$ for any vertex $x$. 

A digraph $\Gamma$ is {\it normally regular with parameters $(n,k,\lambda,\mu)$} if $\Gamma$ is asymmetric, the number of vertices of $\Gamma$ is $n$ and the adjacency matrix $A$ of $\Gamma$ satisfies 
\begin{align}\label{eq:nr}
AA^T=kI_n+\lambda(A+A^T)+\mu(J_n-I_n-A-A^T), 
\end{align}
where $I_n$ is the identity matrix of order $n$ and $J_n$ is the all ones matrix of order $n$. 
It was shown in \cite{JJKS} that a normally regular digraph is indeed normal.
A {\em doubly regular asymmetric digraph} $\Gamma$ with parameters 
$(v,k,\lambda)$ is a normally regular digraph with parameters $(v,k,\lambda,\lambda)$. 

A subset $C$ in $X$ is a {\em coclique} (or an {\em independence set}) in $\Gamma$ if $(x,y)\not\in E$ for any $x,y\in C$.

A digraph $\Gamma$ is {\em strongly connected} if for any distinct vertices $x,y$, there exist vertices $x_0,\ldots,x_s$ such that $x_0=x$, $x_s=y$ and $(x_i,x_{i+1})\in E$ for any $i\in\{0,1,\ldots,s-1\}$.  

The following lemma will be used in Propositin~\ref{prop:sb}. 
\begin{lemma}\label{lem:e}
Let $\Gamma$ be a normally regular digraph with parameters $(n,k,\lambda,\mu)$ with adjacency matrix $A$. 
Assume that $A+A^T=J_n-I_r\otimes J_{n/r}$ for some positive integer $r$ dividing $n$. 
Then the eigenvalues of $A$ are $k$, $\pm\sqrt{-k+\mu}$, or $-n/(2r)\pm\sqrt{k-\mu+(-\lambda+\mu)n/r-n^2/(4r^2)}$.    
\end{lemma} 
\begin{proof}
The valency $k$ is an eigenvalue of $A$ with the all-ones vector as an eigenvector. 
Let $\alpha$ be an eigenvalue whose eigenvector is orthogonal to the all-ones vector. 
By the equation \eqref{eq:nr}, we have 
\begin{align}\label{eq:e1}
\alpha\overline{\alpha}=k+\lambda(\alpha+\overline{\alpha})+\mu(-1-\alpha-\overline{\alpha}). 
\end{align}
Since $A+A^T=J_n-I_r\otimes J_{n/r}$, the real part of $\alpha$ is  $-n/(2r)$ or $0$.
By \eqref{eq:e1}, $\alpha$ is the desired value. 
\end{proof}

\subsection{Association schemes}
A \emph{commutative association scheme of class $d$}
with vertex set $X$ of size $n$ 
is a set of non-zero $(0,1)$-matrices $A_0, \ldots, A_d$, which are called {\em adjacency matrices}, with
rows and columns indexed by $X$, such that:
\begin{enumerate}
\item $A_0=I_n$.
\item $\sum_{i=0}^d A_i = J_n$.
\item For any $i\in\{0,1,\ldots,d\}$, $A_i^T\in\{A_0,A_1,\ldots,A_d\}$.
\item For any $i,j\in\{0,1,\ldots,d\}$, $A_iA_j=\sum_{k=0}^d p_{ij}^k A_k$
for some $p_{ij}^k$'s.
\item For any $i,j\in\{0,1,\ldots,d\}$, $A_iA_j=A_jA_i$. 
\end{enumerate}
The association scheme is said to be \emph{symmetric} if all $A_i$ are symmetric, \emph{nonsymmetric} otherwise.
The \emph{intersection matrix} $B_i$ ($i\in\{0,1,\ldots,d\}$) is $B_i=(p_{ij}^k)_{j,k=0}^{d}$. 

A digraph $\Gamma=(X,E)$ is a \emph{relation graph} of an association scheme with vertex set $X$ if the adjacency matrix of $\Gamma$ is one of the adjacency matrices of the association scheme.  

The vector space spanned by $A_i$'s forms a commutative algebra, denoted by $\mathcal{A}$ and called the \emph{Bose-Mesner algebra} or \emph{adjacency algebra}.
There exists a basis of $\mathcal{A}$ consisting of primitive idempotents, say $E_0=(1/n)J_n,E_1,\ldots,E_d$. 
Since  $\{A_0,A_1,\ldots,A_d\}$ and $\{E_0,E_1,\ldots,E_d\}$ are two bases of $\mathcal{A}$, there exist the change-of-bases matrices $P=(P_{ij})_{i,j=0}^d$, $Q=(Q_{ij})_{i,j=0}^d$ so that
\begin{align*}
A_j=\sum_{i=0}^d P_{ij}E_i,\quad E_j=\frac{1}{n}\sum_{i=0}^d Q_{ij}A_i.
\end{align*}
The matrix $P$ ($Q$ respectively) is said to be the {\em first (second respectively) eigenmatrix}.

\section{Hoffman's bound for normal digraphs}
In this section, we give an upper bound for the size of cocliques in a normal digraph in terms of eigenvalues of the adjacency matrix of the digraph $\Gamma$. 
The upper bound is referred to as the  \emph{Hoffman bound}. 
For a digraph with adjacency matrix $A$,  
define $\theta_{\min}=\min\{\text{Re}(\theta)\mid \theta \text{ is an eigenvalue of }A\}$, $\text{Re}(\theta)$ is the real part of $\theta$. 
Note that for a normal graph,  $\theta_{\min}$ is negative since the trace of $A$ is zero unless $\Gamma$ has no edge.  
\begin{proposition}\label{prop:hb}
Let $n,k$ be positive integers. 
Let $\Gamma=(X,E)$ be a strongly connected $k$-regular normal digraph with $n$ vertices and adjacency matrix $A$. 
For a coclique $C$ in $\Gamma$,  it holds that 
\begin{align}\label{eq:hb}
|C|\leq \frac{n(-\theta_{\min})}{k-\theta_{\min}}.
\end{align} 
Moreover the following hold.
\begin{enumerate}
\item If equality holds in \eqref{eq:hb}, then $|\{y\in C\mid (x,y)\in E\}|+|\{y\in C\mid (y,x)\in E\}|=-2\theta_{\min}$ for any $x\in X\setminus C$.
\item If equality holds in \eqref{eq:hb} and the number of eigenvalues with real part equal to $\theta_{\min}$ is exactly one, then $|\{y\in C\mid (x,y)\in E\}|=-\theta_{\min}$ for any $x\in X\setminus C$.
\end{enumerate}
\end{proposition}
\begin{proof}
Let $\theta_1,\ldots,\theta_{l+2m}$ be the eigenvalues of $A$ such that $\theta_i\in\mathbb{R}$ for any $i\in\{1,\ldots,l\}$ and $\overline{\theta_{l+j}}=\theta_{l+m+j}\not\in\mathbb{R}$ for any $j\in\{1,\ldots,m\}$.
Let $E_i$ be the orthogonal projection onto the eigenspace of $\theta_i$. 
Then $\overline{E_{l+j}}=E_{l+m+j}$ for any $j\in\{1,\ldots,m\}$. 
Since $k$ is an eigenvalue, we set $\theta_1=k$. 
Since $\Gamma$ is strongly connected, $E_1=\frac{1}{n}J_n$. 

Let $\chi$ be the characteristic column vector of $C$. 
Since $C$ is a coclique of $\Gamma$, it holds that 
\begin{align}\label{eq:ind1}
\chi^T A \chi=0.
\end{align} 
On the other hand we estimate the value $\chi^T A \chi$ to use the formula $A=\sum_{i=1}^{l+2m}\theta_iE_i$ as follows:
\begin{align}
\chi^T A \chi&= \chi^T (\sum_{i=1}^{l+2m}\theta_i E_i) \chi= \sum_{i=1}^{l+2m}\theta_i \chi^T E_i \chi \nonumber \\
&=k\chi^T E_1 \chi +\sum_{i=2}^l \theta_i \chi^T E_i \chi +\sum_{i=1}^{2m}\frac{\theta_{l+i}+\overline{\theta_{l+i}} }{2}\chi^T E_{l+i} \chi \nonumber \\
&\geq k\chi^T E_1 \chi+\theta_{\min}\sum_{i=2}^{l+2m} \chi^T E_i \chi \nonumber\\
&=k\chi^T E_1 \chi+\theta_{\min}\chi^T (I_n-E_1) \chi \nonumber \\
&=\frac{(k-\theta_{\min})|C|^2}{n}+\theta_{\min}|C| \label{eq:ind2}
\end{align}
Combining \eqref{eq:ind1} and \eqref{eq:ind2}, we obtain $|C|\leq n(-\theta_{\min})/(k-\theta_{\min})$. 

A coclique $C$ meets the upper bound if and only if $\chi^T E_i \chi=0$  for $i$ such that $i\geq 2$ and $\text{Re}(\theta_i)\neq \theta_{\min}$.

(i): Let $A+A^T=\sum_{i=1}^t \tau_i F_t$ be the spectrum decomposition of $A+A^T$, and set $\tau_1=2k$ and $\tau_t=2\theta_{\min}$. 
Since $F_i\chi=0$ for $i\in\{2,3,\ldots, t-1\}$, 
\begin{align}
(A+A^T)\chi&=2k F_1\chi+\tau_t F_t\chi=2k F_1\chi+\tau_t\sum_{i=2}^t F_i\chi=2k F_1\chi+\tau_t(I_n-F_1)\chi\nonumber\\ 
&=\tau_t \chi+(2k-\tau_t)\frac{1}{n}J_n\chi=\tau_t \chi+(2k-\tau_t)\frac{|C|}{n}\boldsymbol{1}=(-2\theta_{\min})(\boldsymbol{1}-\chi),\label{eq:ind4}
\end{align}
where $\boldsymbol{1}$ is the all-ones vector.
The equation \eqref{eq:ind4} is equivalent to the condition that the $|\{y\in C\mid (x,y)\in E\}|+|\{y\in C\mid (x,y)\in E\}|=-2\theta_{\min}$ for any $x\in X \setminus C$.

(ii):  
Let $\theta_s$ satisfy $\text{Re}(\theta_s)=\theta_{\min}$. Then $\theta_s=\theta_{\min}$. Indeed, if $\theta_s\in\mathbb{C}\setminus\mathbb{R}$, then $\overline{\theta_s}$ also satisfies $\text{Re}(\overline{\theta_s})=\theta_{\min}$. 
This contradicts to the assumption. 
In this case, 
\begin{align}
A\chi&=k E_1\chi+\theta_s E_s\chi=k E_1\chi+\theta_s\sum_{i=2}^sE_i\chi=k E_1\chi+\theta_s(I_n-E_1)\chi\nonumber\\ 
&=\theta_s \chi+(k-\theta_s)\frac{1}{n}J_n\chi=\theta_s \chi+(k-\theta_s)\frac{|C|}{n}\boldsymbol{1}=(-\theta_{\min})(\boldsymbol{1}-\chi).\label{eq:ind3}
\end{align} 
The equation \eqref{eq:ind3} is equivalent to the condition that the size of $\{y\in C\mid (x,y)\in E\}=-\theta_{\min}$ for any $x\in X\setminus C$.
\end{proof}

\begin{remark}\label{rem:1}
Assume that a normally regular digraph $\Gamma$ satisfies the assumptions of Lemma~\ref{lem:e}. 
By $A+A^T=J_n-I_r\otimes J_{n/r}$, the valency $k$ of $\Gamma$ is $\frac{n(r-1)}{2r}$. Thus the right hand side of the bound in  Proposition~\ref{prop:hb} is $n/r$. 
Then the cocliques represented as the main diagonal blocks in $A$ attain the bound in Proposition~\ref{prop:hb}. 
\end{remark}
\section{A characterization of skew-Bush-type Hadamard matrices}\label{sec:sbh}
A {\em Hadamard matrix of order $n$} is an $n\times n$ $(1,-1)$-matrix such that $H H^T=nI_n$. 
A Hadamard matrix $H$ of order $4n^2$ is of {\em Bush-type (or checkered)} if $H=(H_{ij})_{i,j=1}^{2n}$, where $H_{ij}$ is a $2n\times 2n$ matrix for any $i,j\in\{1,\ldots,2n\}$, such that $H_{ii}=J_{2n}$ for any $i\in\{1,\ldots,2n\}$ and  $H_{ij}J_{2n}=J_{2n}H_{ij}=0$ for any distinct $i,j\in\{1,\ldots,2n\}$.
A Bush-type Hadamard matrix $H=(H_{ij})_{i,j=1}^{2n}$ of order $4n^2$ is of {\em skew-Bush-type (or skew-checkered)} if $H-I_{2n}\otimes J_{2n}$ is skew-symmetric.

It was shown by Haemers and Tonchev  in \cite{HT} that some symmetric association scheme of class $3$ exists if and only if  strongly regular graphs with vertex set being decomposed into disjoint cliques attaining Hoffmann's clique bound.
It was shown by Wallis in \cite{W}, that there exists a symmetric Bush-type Hadamard matrix of order $4n^2$ if and only if there exists a strongly regular graph with parameters $(4n^2,2n^2-n,n^2-n,n^2-n)$ such that the vertex set is decomposed into $2n$ disjoint cocliques of size $2n$ (see also \cite[Lemma~1.1]{MX}).

Digraph's counterpart of the result of Haemers and Tonchev by restricting the parameters to $(4n^2,2n^2-n,n^2-n,n^2-n)$ was shown in \cite{GC}, which says there exists some imprimitive nonsymmetric association scheme if and only if there exists a skew-Bush-type Hadamard matrix.  
In this section, we show digraph's counterpart of the result of Wallis \cite{W}, namely characterize the skew-Bush-type Hadamard matrices in terms of the notion of doubly regular asymmetric digraphs with a similar property to the undirected case. 
\begin{proposition}\label{prop:sb}
The following are equivalent.
\begin{enumerate}
\item There exists a skew-Bush-type Hadamard matrix of order $4n^2$.  
\item There exists a doubly regular asymmetric digraph with parameters $(4n^2,2n^2-n,n^2-n)$ such that the vertex set is decomposed into $2n$ disjoint cocliques of size $2n$. 
\end{enumerate}
\end{proposition}
\begin{proof}
(i)$\Rightarrow$(ii): Let $H$ be a skew-Bush-type Hadamard matrix of order $4n^2$.
Define a $(0,1)$-matrix $A=\frac{1}{2}(J_{4n^2}-H)$. 
Since $H-I_{2n}\otimes J_{2n}$ is skew-symmetric, $A$ satisfies that $A+A^T=J_{4n^2}-I_{2n}\otimes J_{2n}$.  
Thus $A$ is the adjacency matrix of a digraph whose vertex set is decomposed into disjoint $2n$ cliques of size $2n$. 
Since $H$ is a regular Hadamard matrix in particular,  it follows that $A$ satisfies the equation $AA^T=n^2I_{4n^2}+(n^2-n)J_{4n^2}$.    
This shows that $A$ is the adjacency matrix of a doubly regular asymmetric digraph with the desired parameters.

(ii)$\Rightarrow$(i): 
Let $\Gamma$ be a doubly regular asymmetric digraph with parameters $(4n^2,2n^2-n,n^2-n)$ with the property that the vertex set is decomposed into $2n$ disjoint cocliques of size $2n$.  
Let $A$ be the adjacency matrix of $\Gamma$. 
Since  $\Gamma$ is decomposed into $2n$ disjoint cocliques of size $2n$, after a suitable rearranging the ordering of the vertices, we may assume that $A+I_{2n}\otimes J_{2n}$ is a $(0,1)$-matrix. 
Let  $H=A-A^T+I_{2n}\otimes J_{2n}$, and set $H_{ij},A_{ij}$ ($i,j\in\{1,\ldots,2n\}$) to be $2n\times 2n$ matrices such that $H=(H_{ij})_{i,j=1}^{2n}$ and $A=(A_{ij})_{i,j=1}^{2n}$. 
Then $H$ is a $(1,-1)$-matrix, and the direct calculation shows that $H$ is a Hadamard matrix. 
It is clear that each diagonal block of size $2n$ is $J_n$ and $H-I_{2n}\otimes J_{2n}$ is skew-symmetric. 
By Lemma~\ref{lem:e} the eigenvalues of $A$ are $2n^2-n,\pm\sqrt{-1}n,-n$.  
As is shown in Remark~\ref{rem:1}, 
the disjoint $2n$ cocliques represented as the main diagonal blocks of $A$ attain the upper bound in Proposition~\ref{prop:hb}, and thus by Proposition~\ref{prop:hb}(ii) we have $A_{ij}J_{2n}=J_{2n}A_{ij}=nJ_{2n}$ for any distinct $i,j$, namely $H_{ij}J_{2n}=J_{2n}H_{ij}=0$. 
Therefore $H$ is a skew-Bush-type Hadamard matrix. 
\end{proof}

\section{Regular biangular matrices and association schemes}\label{sec:HL}
In \cite{HKS} they constructed association schemes from a Hadamard matrix of order $n$ and mutually orthogonal Latin squares of order $n-1$. 
In this section, we construct some association scheme from a Hadamard matrix of order $n$ and a single Latin with some properties of order $n-1$. 
Some relation graphs of the association schemes have the property that its vertex set is decomposed into disjoint cocliques attaining the bound in Proposition~\ref{prop:hb}.

An \emph{$(\alpha,\beta)$-biangular matrix of order $n$} is an $n\times n$ $(1,-1)$-matrix $H$ such that 
the inner products of its normalized rows of $H$ are in $\{\alpha,\beta\}$ \cite{HKS}.
An $(\alpha,\beta)$-biangular matrix $H$ of order $nm$ is called \emph{regular}
if the rows of $H$ can be partitioned into $m$-classes of size $n$ each in such a way that:
\begin{enumerate}
\item $\vert\langle u,v\rangle\vert = \alpha$ for each distinct pair $u,v$ in the same class,
\item $\vert\langle u,v\rangle\vert = \beta$ for each pair $u,v$ belonging to different classes.
\end{enumerate}
We will use the following lemma proven in \cite{K}. 
\begin{lemma}\label{lemma:HCrows}
If there exists a Hadamard matrix, 
then there exist symmetric $(1,-1)$-matrices
$C_1, C_2, \ldots, C_n$ such that:
\begin{enumerate}
\item $C_1 = J_n$. 
\item $C_iC_j = 0$, $1\le i\ne j\le n$.
\item $C_i^2 = nC_i$, $1\le i\le n$.
\item $\sum_{i=1}^n C_i= nI_n$.
\end{enumerate}
It follows from these conditions that the row sums and column sums are $0$ for $C_i$, $i\ne 1$,
and that $$\sum_{i=2}^nC_i^2 = n^2I_n-nJ_n\;.$$
\end{lemma}
\begin{proof}
Letting $H$ be a normalized Hadamard matrix with $i$-th row $h_i$ for $i\in\{1,\ldots,n\}$, set $C_i=h_i^Th_i$. 
Then $C_1,\ldots,C_n$ satisfy the conditions (i)-(iv).  
\end{proof}
Let $H=(H_{ij})_{i,j=1}^n$ be a regular $(\alpha,\beta)$-biangular matrix of order $nm$, where each $H_{ij}$ is a square matrix of order $m$ and  the rows in $i$-th block are in the same class for any $i\in\{1,\dots,m\}$. 
The regular $(\alpha,\beta)$-biangular matrix $H$ is said to be of {\it skew-symmetric} if $H_{ij}^T=-H_{ji}$ for any distinct $i,j\in\{1,\ldots,n\}$. 
\begin{theorem}\label{theorem:bm}
Let $n$ be the order of a Hadamard matrix. 
Then the following hold.
\begin{enumerate}
\item There is a symmetric regular $(0,\frac{1}{n-1})$-biangular matrix of order $n(n-1)$.
\item There is a skew-symmetric regular $(0,\frac{1}{n-1})$-biangular matrix of order $n(n-1)$.
\end{enumerate}
\end{theorem}
\begin{proof}
Let $H$ be a normalized Hadamard matrix, and let $L$ be an addition table of $\mathbb{Z}_{n-1}$. 
Then $L$ is a symmetric Latin square with $(i,j)$-entry denoted by $l(i,j)$. 
We regard $L$ as a Latin square on the set $\{2,\ldots,n\}$.
 
Starting with a symmetric Latin square on the set $\{2,\ldots,n\}$ and
substituting $i$ with $C_i$ from Lemma \ref{lemma:HCrows}
for $i\in\{2,\ldots,n\}$, we obtain a matrix which we will denote by $M$.
Clearly $M$ is a $(1,-1)$-matrix of order $n(n-1)$.
It follows from Lemma \ref{lemma:HCrows} that
$MM^T$ is a block matrix with all diagonal blocks equal to
$n^2I_n-nJ_n$ by Lemma~\ref{lemma:HCrows}, 
and off-diagonal blocks equal to zero matrix by Lemma~\ref{lemma:HCrows} (ii) and the property of $L$ being a Latin square.  
This completes the proof of (i).

For a construction (ii), define $M=(M_{ij})_{i,j=1}^n$ by 
$M_{ij}=C_{l(i,j)}$ for $i\leq j$ and $M_{ij}=-C_{l(i,j)}$ for $i>j$. 
Then it is easy to see that the matrix $M$ is the desired  skew-symmetric biangular matrix. 
\end{proof}

More precisely, the matrices $M$ in Theorem~\ref{theorem:bm} (i), (ii) satisfy the following equation:
\begin{align}\label{eq:3}
MM^T=n(n-1)I_{n(n-1)}-nI_{n-1}\otimes (J_n-I_n).
\end{align}
We decompose $M$ into disjoint $(0,1)$-matrices $A_0,A_1,\ldots,A_4$ defined as follows:
\begin{align}
M&=A_0+A_1-A_2+A_3-A_4\nonumber\\
A_0&=I_{n(n-1)}\nonumber\\
A_1+A_2&=(J_{n-1}-I_{n-1})\otimes J_{n},\label{eq:01}\\
A_0+A_3+A_4&=I_{n-1}\otimes J_{n}\label{eq:02}.
\end{align}
Note that $A_1=A_1^T,A_2=A_2^T$ if $M$ is a symmetric regular biangular matrix and $A_1=A_2^T$ if $M$ is a skew-symmetric regular biangular matrix, and $A_3,A_4$ are symmetric in both cases. 
\begin{theorem}\label{thm:as}
\begin{enumerate}
\item The set of matrices $\{A_0,A_1,A_2,A_3,A_4\}$ forms a symmetric association scheme if $M$ is a symmetric regular biangular matrix. 
\item The set of matrices $\{A_0,A_1,A_2,A_3,A_4\}$ forms a nonsymmetric association scheme if $M$ is a skew-symmetric regular biangular matrix.
\end{enumerate}
\end{theorem}
\begin{proof}
In both cases, the proof is the same as follows. 
Let $\mathcal{A}:=\text{span}_{\mathbb{R}}\{A_0,A_1,\ldots,A_4\}$. 
Since each block matrix of $A_i$ for any $i$ has a constant row and column sum, we have 
\begin{align}
A_i (I_{n-1}\otimes J_{n})&=(I_{n-1}\otimes J_{n}) A_i\in\mathcal{A}\label{eq:1},\\
A_i ((J_{n-1}-I_{n-1})\otimes J_{n})&=((J_{n-1}-I_{n-1})\otimes J_{n}) A_i\in\mathcal{A}\label{eq:2}.
\end{align}

First we show $A_iA_j\in\mathcal{A}$ for $i,j\in\{3,4\}$. 
By Lemma~\ref{lemma:HCrows} (iii), we have $(A_0+A_3-A_4)^2=n(A_0+A_3-A_4)$.  
By \eqref{eq:02} and \eqref{eq:1} we have $A_4^2\in\mathcal{A}$. 
Similarly $A_iA_j\in\mathcal{A}$ for others $i,j\in\{3,4\}$.

Next we show $A_iA_j\in\mathcal{A}$ for $i\in\{1,2\},j\in\{3,4\}$ or $i\in\{3,4\},j\in\{1,2\}$.
By Lemma~\ref{lemma:HCrows} (ii), we have $(A_1-A_2)(A_0+A_3-A_4)=0$. 
By \eqref{eq:01}, \eqref{eq:02}, \eqref{eq:1} and \eqref{eq:2}, we have $A_2A_4\in\mathcal{A}$. 
Similarly $A_iA_j\in\mathcal{A}$ for others $i\in\{1,2\},j\in\{3,4\}$ or $i\in\{3,4\},j\in\{1,2\}$.

Finally we show $A_iA_j\in\mathcal{A}$ for $i,j\in\{1,2\}$.
 By \eqref{eq:3} we have $(A_1-A_2)^2\in\mathcal{A}$. 
By \eqref{eq:01} and \eqref{eq:2}, we have $A_iA_j\in\mathcal{A}$ for $i,j\in\{1,2\}$.  
Thus this completes the proof. 
\end{proof}
The first eigenmatrices in Theorem~\ref{thm:as} (i), (ii) are as follows respectively:
\begin{align*}
P&=\begin{pmatrix}
1 & \frac{n(n-2)}{2} & \frac{n(n-2)}{2} & \frac{n-2}{2} & \frac{n}{2} \\
1 & 0 & 0 & \frac{n-2}{2} & -\frac{n}{2} \\
1 & -\frac{n}{2} & -\frac{n}{2} & \frac{n-2}{2} & \frac{n}{2} \\
1 & -\frac{n}{2} & \frac{n}{2} & -1 & 0 \\
1 & \frac{n}{2} & -\frac{n}{2} & -1 & 0 
\end{pmatrix},
P=\begin{pmatrix}
1 & \frac{n(n-2)}{2} & \frac{n(n-2)}{2} & \frac{n-2}{2} & \frac{n}{2} \\
1 & 0 & 0 & \frac{n-2}{2} & -\frac{n}{2} \\
1 & -\frac{n}{2} & -\frac{n}{2} & \frac{n-2}{2} & \frac{n}{2} \\
1 & -\frac{\sqrt{-1}n}{2} & \frac{\sqrt{-1}n}{2} & -1 & 0 \\
1 & \frac{\sqrt{-1}n}{2} & -\frac{\sqrt{-1}n}{2} & -1 & 0 
\end{pmatrix}.
\end{align*}
See Appendices A, B for the intersection numbers and second eigenmatrices. 
Consider relation graphs with adjacency matrix $A_1,A_2$ in both association schemes.   
As Proposition~\ref{prop:hb} shows, each main diagonal block of $A_1,A_2$ represents a coclique and $A_4$ corresponds to a partition of the vertex set by cliques attaining the bound in Proposition~\ref{prop:hb}.

\section{Twin asymmetric designs and association schemes}\label{sec:td}
Finally we focus on normally regular digraphs with $\lambda=\mu$, or equivalently doubly regular asymmetric graphs. 
If an incidence matrix $N$ of a symmetric design is such that 
$N+N^{T}$ is a $(0,1)$-matrix, then $N$ is an adjacency matrix of a 
doubly regular asymmetric digraph, and vice versa. Our main reference for 
this section is \cite{yh-drad}.  
We will refer to a doubly regular asymmetric digraph with parameters 
$(v,k,\lambda)$ as a $DRAD(v,k,\lambda)$.
Symmetric $(v,k,\lambda)$-designs $\mathbf{D}=(X,\mathcal{B})$ and 
$\mathbf{D}^{\prime}=(X,\mathcal{B}^{\prime})$ are called {\em twin designs} 
if there is a bijection $f\colon \mathcal{B}\to\mathcal{B}^{\prime}$ such that 
every block $B\in\mathcal{B}$ is disjoint from $f(B)$.  In general, it is 
not easy to find twin symmetric designs.  However, if $\Gamma$ is a 
$DRAD(v,k,\lambda)$ and $\Gamma^{\prime}$ is the digraph obtained by 
reversing the direction of every arc of $\Gamma$, then the 
corresponding symmetric designs are twins.
The following theorem is proven in \cite{yh-drad}.
\begin{theorem}\label{RHDRAD}
Let $h$ be a positive integer such that there exists a Hadamard matrix of 
order $2h$.  If $p=(2h-1)^{2}$ is a prime power, then, for any 
positive integer $d$, there exists a
\begin{equation}\label{1} 
DRAD\left(\frac{h(p^{2d}-1)}{h+1},hp^{2d},h(h+1)p^{2d-1}\right).
\end{equation}
\end{theorem}

The construction makes use of {\em skew balanced generalized weighing matrices}
and Bush-type Hadamard matrices constructed as in Lemma~\ref{lemma:HCrows} from a Hadamard matrix of 
order $2h$. We illustrate this by an example which relates to the special case of the theorem which used in this note.

\begin{example}\label{drad160}
We start with a BGW$(10,9,8)$ over the cyclic group $C_8$.
Let \[ W=[w_{ij}]=\left(\begin{array}{cccccccccc}
     0&1&1&1&1&1&1&1&1&1\\
 4&0&3&7&5&6&8&1&4&2\\
 4&7&0&3&8&5&6&2&1&4\\
 4&3&7&0&6&8&5&4&2&1\\
 4&1&4&2&0&3&7&5&6&8\\
 4&2&1&4&7&0&3&8&5&6\\
 4&4&2&1&3&7&0&6&8&5\\
 4&5&6&8&1&4&2&0&3&7\\
 4&8&5&6&2&1&4&7&0&3\\
 4&6&8&5&4&2&1&3&7&0
  \end{array}\right).
    \]
Then $W$ is a skew BGW$(10,9,8)$ over the cyclic group $C_8=\langle g\rangle$ generated by the matrix
\[
 g=\left(\begin{array}{cccc}
0&I_4&0&0\\
0&0&I_4&0\\
0&0&0&I_4\\
-I_4&0&0&0
\end{array}\right),
\]
where 
the number $i$ in $G$ denotes $g^i$ for $i=1,2,\ldots,8$. Let
\[
 H=\left(\begin{array}{cccc}
       0&C_2&  C_3& C_4\\
-C_4&0&  C_2& C_3\\
-C_3&-C_4&  0& C_2\\
-C_2&-C_3&  -C_4& 0
         \end{array}\right),
\]
where $C_2,C_3,C_4$ are those constructed in \ref{lemma:HCrows} from
a normalized Hadamard matrix of order 4 and $0$ denotes the zero matrix of order 16.\\

Let 
\[
 R=\left(\begin{array}{cccc}
0&0&0&I_4\\
0&0&I_4&0\\
0&I_4&0&0\\
I_4&0&0&0
\end{array}\right).
\]

Let $G=[Hw_{ij}R]$, then $G$ can be splitted to parts, namely the positive and negative part, to form a twin skew symmetric $(160,54,18)$ design on 160 vertices. 

To do this, keep all the $1$-entries in $G$, 
change all the $-1$-entries to $0$ and let $A_1$ be the $(0,1)$-matrix obtained.
Then,  $A_1$ is the incidence matrix of a symmetric $(160,54,18)$ design. Furthermore,
$A_1+A_1^T$ is a $(0,1)$-matrix. So, $A_1$ is the adjacency matrix of a doubly regular asymmetric digraph.
Now  
change all the $1$-entries in $G$ to $0$, all $-1$-entries to $1$ and let $A_2$ be the $(0,1)$-matrix obtained. 
Then $A_2=A_1^T$, so $A_1$ and $A_2$ are twins. We refer the reader to \cite{yh-drad} for the general construction.
\end{example}

We now use the sequence of doubly regular digraphs obtained from the above theorem 
for $d=1$ to deduce the existence of some association schemes of class five. The general 
case corresponding to any positive integer $d$ will appear elsewhere.

\begin{theorem}
Let $h=2n$ be a positive integer for which there is a Hadamard matrix of order $h$ and $p=2n-1$ is a prime power. 
Consider the skew\\ $BGW(p^2+1,p^2,p^2-1)$ over the cyclic group of order $4n$ and the twin design constructed in \cite{yh-drad} for $d=1$.\\
Let $A_1$ be the plus and $A_2$ the minus twin, $A_4=I_{2n(p^2+1)}\otimes (J_{2n}-I_{2n})$, $A_5=I_{p^2+1}\otimes (J_{4n^2}-I_{2n}\otimes J_{2n})$.\\
Then $\{A_0=I_{4n^2(p^2+1)},A_1,A_2,A_3=J_{4n^2(p^2+1)}-A_1-A_2-A_4-A_5, A_4,A_5\}$\\ forms a nonsymmetric association scheme of class 5 with
the following intersection numbers. Note that $A_1^T=A_2$, $A_3,A_4,A_5$ are symmetric.
\begin{itemize}
 \item 
$A_1A_1 = A_2A_2 =(n-1)(2n-1)(2n^2-n)(A_1+A_2+A_3+A_5)+ n^2(2n-1)^2A_4.$
\item
$A_1A_2 =n^2(2n-1)^2A_0+(2n-1)^2(n^2-n)J.$
\item
$A_1A_3 = A_2A_3=2n(n-1)(2n-1)(A_1+A_2+A_3)+n(2n-1)^2 A_5.$
\item
$A_1A_4 = (n-1)A_1 + nA_2.$
\item
$A_1A_5 = A_2A_5=2n(n-1)(A_1+A_2) +n(2n-1)A_3.$
\item
$A_2A_4  =  nA_1+ (n-1)A_2.$ 
\item
$A_3A_3 =2n(2n-1)^2A_0+4n(n-1)(A_1+A_2+A_3)+2n(2n-1)^2A_4.$
\item
$A_3A_4  = (2n-1)A_3.$
\item
$A_3A_5  =2n(A_1+A_2).$ 
\item
$A_4A_4  =(2n-1)A_0 + (2n-2)A_4.$\item 
$A_4A_5  = (2n-1)A_5.$
\item
$A_5A_5 = 2n(2n-1)A_0+2n(2n-1)A_4+4n(n-1)A_5.$
\end{itemize}
\end{theorem}
\begin{proof}
 Let $W=[w_{ij}]$ be a skew $BGW(p^2+1,p^2,p^2-1)$ over a cyclic group of order $4n$ generated by a negacirculant matrix
of order $4n$ as described in \cite{yh-drad}. Let $R=R_{2n}\otimes I_{2n}$, where $R_{2n}$ denotes the back identity matrix of order $2n$ and $I_{2n}$ is the identity matrix of order $2n$. Then $A_3=[|w_{ij}|R]$. \\
The identities for $A_1A_2=A_2A_1$ follows from the fact that each of $A_1$ and $A_2$ are the inicdence matrices ofa symmetric
$(p^2+1)4n^2, p^2(2n^2-n),p^2(n^2-n))$ designs and $A_1^T=A_2$. The numbers for $A_1A_1$ and $A_2A_2$ follows from 
the fact that the symmetric matrix $A_1+A_2$ has a simple structure and we make use of it in finding the numbers for other
products involving $A_1$ and $A_2$. The relation related to $A_3$ follwos from the observation that $A_3=[|w_{ij}|R]$ and $A_1+A_2+A_3=J_{4n^2(p^2+1)}-I_{p^2+1}\otimes J_{4n^2}$. The ramining numbers are not hard to calculate. 
\end{proof}

The eigenmatrices $P,Q$ are given as follows:
\begin{align*}
P&=\begin{pmatrix}
1 & n(2n-1)^3 & n(2n-1)^3 & 2n(2n-1)^2 & 2n-1 & 2n(2n-1) \\
1 & n(2n-1)    & n(2n-1)    & -2n(2n-1)  & 2n-1 & -2n \\
1 & n(2n-1)\sqrt{-1}    & -n(2n-1)\sqrt{-1}    & 0  & -1 & 0 \\
1 & -n(2n-1)\sqrt{-1}  & n(2n-1)\sqrt{-1}    & 0  & -1 & 0 \\
1 & -n(2n-1)   & -n(2n-1)    & -2n  & 2n-1 & 2n(2n-1) \\
1 & -n(2n-1)   & -n(2n-1)    & -2n(2n-1)  & 2n-1 & -2n 
\end{pmatrix},\\
Q&=\begin{pmatrix}
1 & (2n-1)m & 2n(2n-1)m & 2n(2n-1)m & (2n-1)^2 & (2n-1)m \\
1 & \frac{m}{2n-1} & -\frac{2nm\sqrt{-1}}{2n-1} & \frac{2nm\sqrt{-1}}{2n-1} & -1  & -\frac{m}{2n-1} \\
1 & \frac{m}{2n-1} & \frac{2nm\sqrt{-1}}{2n-1} & -\frac{2nm\sqrt{-1}}{2n-1} & -1  & -\frac{m}{2n-1}\\
1 & -m & 0 & 0 & -1 & -2n+1\\
1 & m & -2nm & -2nm & (2n-1)^2 & m \\
1 & -m & 0 & 0 & (2n-1)^2 & -m
\end{pmatrix},
\end{align*}
where $m=2n^2-2n+1$. 
By the definition of $A_4,A_5$, we have $A_4+A_5=I_{p^2+1}\otimes(J_{4n^2}-I_{4n^2})$.  
The cocliques of the digraphs whose adjacency matrices are $A_1,A_2$ corresponding to the main diagonal blocks of $A_4+A_5$ attain the upper bound in Proposition~\ref{prop:hb}.
\section*{Acknowledgement}
Hadi Kharaghani is supported by an NSERC Discovery Grant.  
Sho Suda is supported by JSPS KAKENHI Grant Number 15K21075.


\appendix
\def\thesection{Appendix \Alph{section}}

\section{Parameters of the association scheme in Theorem~\ref{thm:as}(i)}
\begin{align*}
B_1&=\begin{pmatrix}
0 & 1 & 0 & 0 & 0 \\
\frac{n^2-2n}{2} & \frac{n^2-3n}{4} & \frac{n^2-3n}{4} & \frac{n^2-4n}{4} & \frac{n^2-2n}{4} \\
0 & \frac{n^2-3n}{4} & \frac{n^2-3n}{4} & \frac{n^2}{4} & \frac{n^2-2n}{4} \\
0 & \frac{n}{4}-1 & \frac{n}{4} & 0 & 0 \\
0 & \frac{n}{4} & \frac{n}{4} & 0 & 0 
\end{pmatrix}\\
B_2&=\begin{pmatrix}
0 & 0 & 1 & 0 & 0 \\
0 & \frac{n^2-3n}{4} & \frac{n^2-3n}{4} & \frac{n^2}{4} & \frac{n^2-2n}{4} \\
\frac{n^2-2n}{2} & \frac{n^2-3n}{4} & \frac{n^2-3n}{4} & \frac{n^2-4n}{4} & \frac{n^2-2n}{4} \\
0 & \frac{n}{4} & \frac{n-4}{4} & 0 & 0 \\
0 & \frac{n}{4} & \frac{n}{4} & 0 & 0 
\end{pmatrix}\\
B_3&=\begin{pmatrix}
0 & 0 & 0 & 1 & 0 \\
0 & \frac{n}{4}-1 & \frac{n}{4} & 0 & 0 \\
0 & \frac{n}{4} & \frac{n}{4}-1 & 0 & 0 \\
\frac{n}{2}-1 & 0 & 0 & \frac{n}{2}-2 & 0 \\
0 & 0 & 0 & 0 & \frac{n}{2}-1 
\end{pmatrix}\\
B_4&=\begin{pmatrix}
0 & 0 & 0 & 0 & 1 \\
0 & \frac{n}{4} & \frac{n}{4} & 0 & 0 \\
0 & \frac{n}{4} & \frac{n}{4} & 0 & 0 \\
0 & 0 & 0 & 0 & \frac{n}{2}-1 \\
\frac{n}{2} & 0 & 0 & \frac{n}{2} & 0
\end{pmatrix}\\
Q&=\begin{pmatrix}
1 & n-1 & n-2 & \frac{(n-1)(n-2)}{2} & \frac{(n-1)(n-2)}{2} \\
1 & 0 & -1 & -\frac{n-1}{2} & \frac{n-1}{2} \\
1 & 0 & -1 & \frac{n-1}{2} & -\frac{n-1}{2} \\
1 & n-1 & n-2 & -n+1 & -n+1 \\
1 & -n+1 & n-2 & 0 & 0  
\end{pmatrix}
\end{align*}

\section{Parameters of the association scheme in Theorem~\ref{thm:as}(ii)}\begin{align*}
B_1&=\begin{pmatrix}
0 & 1 & 0 & 0 & 0 \\
0 & \frac{n^2-3n}{4} & \frac{n^2-3n}{4} & \frac{n^2}{4} & \frac{n^2-2n}{4} \\
\frac{n^2-2n}{2} & \frac{n^2-3n}{4} & \frac{n^2-3n}{4} & \frac{n^2-4n}{4} & \frac{n^2-2n}{4} \\
0 & \frac{n}{4}-1 & \frac{n}{4} & 0 & 0 \\
0 & \frac{n}{4} & \frac{n}{4} & 0 & 0 
\end{pmatrix}\\ 
B_2&=\begin{pmatrix}
0 & 0 & 1 & 0 & 0 \\
\frac{n^2-2n}{2} & \frac{n^2-3n}{4} & \frac{n^2-3n}{4} & \frac{n^2-4n}{4} & \frac{n^2-2n}{4} \\
0 & \frac{n^2-3n}{4} & \frac{n^2-3n}{4} & \frac{n^2}{4} & \frac{n^2-2n}{4} \\
0 & \frac{n}{4} & \frac{n-4}{4} & 0 & 0 \\
0 & \frac{n}{4} & \frac{n}{4} & 0 & 0 
\end{pmatrix}\\
B_3&=\begin{pmatrix}
0 & 0 & 0 & 1 & 0 \\
0 & \frac{n}{4}-1 & \frac{n}{4} & 0 & 0 \\
0 & \frac{n}{4} & \frac{n}{4}-1 & 0 & 0 \\
\frac{n}{2}-1 & 0 & 0 & \frac{n}{2}-2 & 0 \\
0 & 0 & 0 & 0 & \frac{n}{2}-1 
\end{pmatrix}\\
B_4&=\begin{pmatrix}
0 & 0 & 0 & 0 & 1 \\
0 & \frac{n}{4} & \frac{n}{4} & 0 & 0 \\
0 & \frac{n}{4} & \frac{n}{4} & 0 & 0 \\
0 & 0 & 0 & 0 & \frac{n}{2}-1 \\
\frac{n}{2} & 0 & 0 & \frac{n}{2} & 0
\end{pmatrix}\\
Q&=\begin{pmatrix}
1 & n-1 & n-2 & \frac{(n-1)(n-2)}{2} & \frac{(n-1)(n-2)}{2} \\
1 & 0 & -1 & -\frac{\sqrt{-1}(n-1)}{2} & \frac{\sqrt{-1}(n-1)}{2} \\
1 & 0 & -1 & \frac{\sqrt{-1}(n-1)}{2} & -\frac{\sqrt{-1}(n-1)}{2} \\
1 & n-1 & n-2 & -n+1 & -n+1 \\
1 & -n+1 & n-2 & 0 & 0  
\end{pmatrix}
\end{align*}


\begin{thebibliography}{99}
\bibitem{B}
A. E. Brouwer, 
Distance regular graphs of diameter $3$ and strongly regular graphs, 
Discrete Math, 49 (1984), 101--103.

\bibitem{BH}
A. E. Brouwer and W. H.  Haemers, Spectra of graphs. Universitext. {\sl Springer, New York}, 2012. xiv+250 pp.

\bibitem{C}
Y. Chang, Imprimitive Symmetric Association Schemes of Rank 4, Thesis, University of Michigan, 1994.

\bibitem{GC}
R. W. Goldbach and H.L. Claasen, 
$3$-class association schemes and Hadamard matrices of a certain block form, 
{\it Europ. J. Combin.} (1998) 19, 943--951.

\bibitem{HT}
W. H. Haemers and V. D. Tonchev, 
Spreads in strongly regular graphs, 
{\it Des.\ Codes and Crypt.} (1996) 8, 145--157.


\bibitem{HKS}
W.H. Holzmann, H. Kharaghani and S.Suda
Mutually unbiased biangular vectors and association schemes,
to appear in Springer Proceedings in Mathematics and Statistics: Algebraic Design Theory and Hadamard Matrices.

\bibitem{yh-drad}
Y.J. Ionin and H. Kharaghani. Doubly regular digraphs and symmetric designs, 
{\sl J. Combin. Theory Ser. A} 101 (2003), no. 1, 35--48.


\bibitem{JJKS}
L. K. Jorgensen, G. A. Jones, M. H. Klin and S. Y. Song, 
Normally regular digraphs, association schemes and related combinatorial structures, 
{\sl S\'{e}m. Lothar. Combin}. 71 (2013/14), Art. B71c, 39 pp.

\bibitem{K}
H. Kharaghani, 
New class of weighing matrices, 
{\it Ars. Combin.} 19 (1985), 69-72.

\bibitem{KSS}
H. Kharaghani, S. Sasani and S. Suda, 
Mutually unbiased Bush-type Hadamard matrices and association schemes, 
{\sl Elec.\ J. Combin.} 22  (2015) P3. 10.

\bibitem{MX}
M. Muzychuk and Q. Xiang, 
Symmetric Bush-type Hadamard matrices of order $4m^4$ exist for all odd $m$,
{\it Proc. Amer. Math. Soc.} 134 (2006), no. 8, 2197--2204. 

\bibitem{W}
W. D. Wallis, 
On a problem of K. A. Bush concerning Hadamard matrices, 
{\sl Bull.\ Aust.\ Math.\ Soc.} {\bf 6}  (1971) 321--326.
\end{thebibliography}
\end{document}